\documentclass{amsart}

\usepackage[utf8]{inputenc}
\usepackage{amsfonts,mathtools,amsmath,amsthm,amssymb,tikz,color,enumitem,latexsym,tikz-cd,mathrsfs,graphicx,verbatim}
\usetikzlibrary{positioning}
\usepackage{pgfplots, pgfplotstable}
\pgfplotsset{compat=1.14}
\usepgfplotslibrary{statistics}
\usepackage[colorlinks=true]{hyperref}
\usepackage[maxbibnames=99]{biblatex}
\usepackage[algoruled,linesnumbered]{algorithm2e}
\bibliography{sources}
\allowdisplaybreaks

\hypersetup{
    colorlinks,
    linkcolor={red!50!black},
    citecolor={blue!50!black},
    urlcolor={blue!80!black}
}

\usetikzlibrary{arrows}
  
\newtheorem{lem}{Lemma}[section]

\newtheorem{thm}[lem]{Theorem}
\newtheorem{prop}[lem]{Proposition}

\newtheorem{cor}[lem]{Corollary}
\newtheorem{defn}[lem]{Definition}

\theoremstyle{definition}

\newtheorem{rmk}[lem]{Remark}

\def\CC{\mathbb C}

\def\RR{\mathbb R}
\def\ZZ{\mathbb Z}
\def\QQ{\mathbb Q}

\def\FF{\mathbb F}

\def\SpM{{\mathcal S}_M}
\def\Sp{\mathcal S}
\def\quadO{\mathcal O}
\def\QuatO{\mathfrak O}
\DeclareMathOperator\GL{\mathsf{GL}}

\DeclareMathOperator\End{End}
\DeclareMathOperator\trd{trd}
\DeclareMathOperator\N{N}
\DeclareMathOperator\nrd{nrd}
\DeclareMathOperator\disc{disc}

\DeclareMathOperator\Cl{Cl}

\newcommand{\legendre}[2]{\left(\frac{#1}{#2}\right)}

\makeatletter
\def\paragraph{
	\addvspace{\medskipamount}%
	\@startsection{paragraph}{4}%
	\z@\z@{-\fontdimen2\font}%
	{\normalfont\bfseries}}
\makeatother

\usepackage{comment}
\includecomment{shortversion}
\excludecomment{fullversion}
\includecomment{namedversion}
\excludecomment{anonversion}

\begin{document}
	
	\title{Supersingular Curves With Small Non-integer Endomorphisms}
	\begin{namedversion}
		\author{Jonathan Love}
		\address[Jonathan Love]{Stanford University, Dept.\ of Mathematics}
		\email{jonlove@stanford.edu}
		\author{Dan Boneh}
		\address[Dan Boneh]{Stanford University, Dept.\ of Computer Science}
		\email{dabo@cs.stanford.edu}
		\thanks{Supported by NSF grant \#1701567}
		\date{February 2020}
	\end{namedversion}
	
	\begin{abstract}
	We introduce a special class of supersingular curves over $\FF_{p^2}$, characterized by the existence of non-integer endomorphisms of small degree. A number of properties of this set is proved. Most notably, we show that when this set partitions into subsets in such a way that curves within each subset have small-degree isogenies between them, but curves in distinct subsets have no small-degree isogenies between them. Despite this, we show that isogenies between these curves can be computed efficiently, giving a technique for computing isogenies between certain prescribed curves that cannot be reasonably connected by searching on $\ell$-isogeny graphs.
	\end{abstract}

	\maketitle

	\section{Introduction}

Given an elliptic curve $E$ over a field $F$, let $\End(E)$ denote the ring of endomorphisms of $E$ that are defined over $\overline{F}$. The curve $E$ is \textbf{supersingular} 
if $\End(E)$ is non-abelian; this can only occur if $E$ is defined over $\FF_{p^2}$ for some prime $p$~\cite[Theorem V.3.1]{silverman}. While the set of all supersingular curves can be quite complicated, in this paper we define collections of supersingular curves which are relatively easy to compute with and to classify.

\begin{defn}\label{isogdist}
	Given $M<p$, an elliptic curve $E$ over a finite field of characteristic $p$ is \textbf{$M$-small} (we also say that the $j$-invariant of $E$ is $M$-small) if there exists $\alpha\in\End(E)$ with $\deg\alpha\leq M$ such that $\alpha$ is not multiplication by an integer. The set of $M$-small $j$-invariants of \emph{supersingular} curves over $\FF_{p^2}$ is denoted $\SpM$ (with the prime $p$ being assumed from context).
\end{defn}

\noindent
An $M$-small curve may be ordinary or supersingular. This paper will focus primarily on the set of $M$-small supersingular curves, though some results will hold for any $M$-small elliptic curve. Assuming for the rest of this paper that $p\geq 5$, a few notable properties that will be discussed are as follows:

\begin{enumerate}[label=(\alph*)]

	\item\label{ptgen} The set of all $M$-small curves in characteristic $p$ can be generated by finding roots of Hilbert class polynomials for orders of discriminant $O(M)$ (Proposition~\ref{msmallphi}).
	
	\item\label{ptpartition} If $M<\frac{\sqrt{p}}{2}$, the set $\SpM$ of $M$-small supersingular curves partitions into $O(M)$ subsets, each connected by small-degree isogenies, such that there is no isogeny of degree less than $\frac{\sqrt{p}}{2M}$ between distinct subsets (Theorem~\ref{clusterthm}).
	
	\item\label{ptcomp} The endomorphism rings of $M$-small supersingular curves, and isogenies between any two of them, can heuristically be computed in time polynomial in $M$ and $\log p$ (Section~\ref{computations}).
\end{enumerate}
A number of other properties are discussed in an appendix:
\begin{enumerate}[label=(\alph*),resume]
	\item The number of $M$-small curves up to $\overline{\FF_p}$-isomorphism is $O(M^{3/2})$.

	\item When $M\ll p$, approximately half of all $M$-small curves appear to be supersingular (heuristically and experimentally).

	\item\label{ptall} When $M\geq \frac12p^{2/3}+\frac14$, every supersingular curve is $M$-small.
\end{enumerate}

Let us state point~\ref{ptpartition} more precisely. Given an elliptic curve $E$ over $\FF_{p^2}$, let $E^{(p)}$ denote its image under the $p^\text{th}$ power Frobenius map $(x,y)\mapsto (x^p,y^p)$. If $E$ is defined over $\FF_p$, then $E= E^{(p)}$; otherwise we have $E=(E^{(p)})^{(p)}$ and so this map will swap conjugate pairs of curves.\footnote{The map $E\to E^{(p)}$ on supersingular curves is called the ``mirror involution'' in~\cite{arpin}, where the relationship between conjugate pairs, along with many other structural properties of supersingular isogeny graphs, is studied in detail.} For $j\in\FF_{p^2}$, let $E_j$ be an elliptic curve over $\FF_{p^2}$ with $j$-invariant equal to $j$.

\begin{defn}
	Let $E$ and $E'$ be supersingular elliptic curves over $\FF_{p^2}$. The \textbf{distance from $E$ to $E'$}, denoted $d(E,E')$, is the minimum degree of an isogeny $E\to E'$ or $E\to E'^{(p)}$ defined over $\overline{\FF_p}$. We also define $d(j,j')=d(E_j,E_{j'})$ for supersingular $j$-invariants $j,j'\in\FF_{p^2}$.
\end{defn}

\noindent
By basic properties of isogenies (e.g.,~\cite[Chapter III]{silverman}), $\log d$ is a pseudometric on the set of supersingular curves over $\FF_{p^2}$, and it descends to a metric on the set of Galois orbits $\{E,E^{(p)}\}$.

\begin{thm}\label{clusterthm}
	Suppose $p>4M^2$, and let $\SpM$ denote the set of $M$-small supersingular curves. Then there exists a partition
	\[\SpM=\bigsqcup_D T_D\]
	of $\SpM$ into nonempty subsets, indexed by fundamental discriminants $-4M\leq D<0$ which are not congruent to a square mod $p$. This partition has the following properties:
	\begin{enumerate}[label=(\alph*)]
		\item If $j,j'$ are in distinct subsets $T_D\neq T_{D'}$, then 
		\[d(j,j')\geq \frac{\sqrt{p}}{2M}.\]
		\item If $j,j'$ are in the same subset $T_D$, then there is a chain 
		$j=j_0,j_1,\ldots,j_r=j'$ of elements of $T_D$ such that 
		\[d(j_{i-1},j_i)\leq \frac{4}{\pi}\sqrt{M}\]
		for all $i=1,\ldots,r$. We can find such a chain with $r\leq 3$, or alternatively, we can find such a chain such that for each $i=1,\ldots, r$, there exists an isogeny $E_{j_{i-1}}\to E_{j_i}$ or $E_{j_{i-1}}\to E_{j_i}^{(p)}$ with prime degree at most $\frac{4}{\pi}\sqrt{M}$.
	\end{enumerate}
\end{thm}

\begin{figure}
    \centering
    \includegraphics[width=0.75\textwidth]{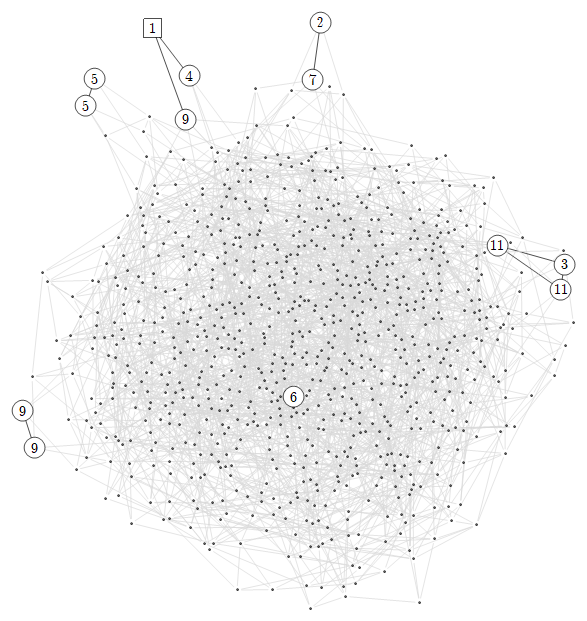}
    \caption{A graph which illustrates Theorem~\ref{clusterthm}. The vertices are supersingular elliptic curves in characteristic $p=20011$, with conjugate pairs $\{E,E^{(p)}\}$ identified. The $12$-small curves are highlighted, and labelled with the smallest degree of a non-integer endomorphism. The square vertex is the curve $y^2=x^3+x$ with $j$-invariant $1728$. 
    Two curves $E,E'$ are connected by an edge if there is an isogeny $E\to E'$ of degree $2$ or $3$ (the primes less than $\frac{4}{\pi}\sqrt{12}$). The connected components of the $M$-small subgraph correspond to the sets $T_D$ for $D=-4,-7,-11,-24,-35,-20$ (starting from the square and proceeding clockwise).
    Data computed using Magma~\cite{magma}, plotted using Mathematica~\cite{mathematica}.}
    \label{fig:SM}
\end{figure}

See Figure \ref{fig:SM} for an illustration of Theorem~\ref{clusterthm}. Intuitively, this is saying that the set of supersingular curves has ``isogeny valleys''\footnote{Perhaps they should be called ``isogeny peaks'' because we shall see in Section~\ref{embeddingorders} that they are very closely related to the volcanic ``craters'' of ordinary isogeny graphs, as discussed in~\cite{suth_isogvolc}. However, it feels more natural to associate $M$-small curves with valleys, both so that we can think of endomorphism degree as a measure of height, and because they are in practice easier to reach, as discussed in Section~\ref{generatingmsmall}.} indexed by certain fundamental discriminants; each valley consists of a number of $M$-small curves that are all linked together by low-degree isogenies, but are very far away from the $M$-small curves in other isogeny valleys. The sizes and shapes of these valleys are discussed in Appendix~\ref{appendixcount}.

The fact that the sets $T_D$ are connected by small-degree isogenies (as described in Theorem~\ref{clusterthm}(b)) will not be evident in the $\ell$-isogeny graph for any individual prime $\ell$. In fact, if $\ell$ is any prime such that one of the sets $T_D$ contains an $(M/\ell^2)$-small curve, then there are two curves in $T_D$ such that the degree of any isogeny between them is either divisible by $\ell$ or greater than $p\ell/(4M)$ (Corollary~\ref{isogwithoutl}). So if we exclude any sufficiently small prime, Theorem~\ref{clusterthm}(b) does not hold.

\paragraph{Motivation.}
We say that a supersingular elliptic curve $E$ over $\FF_{p^2}$ 
is ``hard''
if it is computationally infeasible to compute
its endomorphism ring. 
A number of applications in cryptography (e.g.,~\cite{VDF}) need
an explicit hard curve $E$
where no one, including the party who generated the curve, can compute its endomorphism ring.
Currently, there is no known method to generate such a curve.

To illustrate the problem, suppose $p \equiv 2 \bmod 3$ 
and let $E_0$ be the supersingular curve with $j$-invariant~$0$.
Let $\ell$ be a small prime.
One can generate a large number of supersingular curves by taking a random walk along the graph of degree $\ell$ isogenies, starting at $E_0$.
However, every curve $E$ generated this way will have a known endomorphism ring: the endomorphism ring of $E$ can be computed using the isogeny path from $E_0$ to $E$. 

Point~\ref{ptgen} raises the possibility of using 
the set of $M$-small supersingular elliptic curves, for some polynomial size~$M$, 
as a candidate set of explicit hard curves. 
If $E$ is a typical $M$-small curve, then point~\ref{ptpartition} tells us that $E$ could not reasonably be found by searching from $E_0$ on $\ell$-isogeny graphs for any small primes $\ell$. A priori, this might suggest that it would be difficult to compute the isogeny path from $E_0$ to $E$,
and therefore there is hope that the endomorphism ring of $E$ will remain unknown. 
However, point~\ref{ptcomp} demonstrates that this is likely not the case. 

This suggests that a hard curve will not be $M$-small; by the classification results of Section \ref{modroots}, this rules out roots of low-degree Hilbert class polynomials as reasonable candidates for hard curves.
It remains an open problem to construct a single explicit hard supersingular curve.

\paragraph{Organization.}
The content of this paper is as follows. In Section~\ref{generatingmsmall}, we note that several known examples of supersingular curves are in fact $M$-small for very small values of $M$, and show that an algorithm due to Br\"{o}ker used to generate supersingular curves will typically output $M$-small curves. We will then see how to generate all such curves by generalizing Br\"{o}ker's algorithm.

Sections~\ref{quatbackground}--\ref{msmallapplication} are devoted to the proof of Theorem~\ref{clusterthm}. 
This proof depends on the fact that the endomorphism ring of a supersingular curve is a maximal order in a quaternion algebra,\footnote{This viewpoint lays the foundation for many prior papers on supersingular isogenies; see for instance~\cite{quatisogpath},~\cite{hardeasy}, and~\cite{isoggraphs_endorings}.} so a brief review of some necessary background is given in Section~\ref{quatbackground}. 
In Section~\ref{quatproofs} we lay the groundwork for a proof of Theorem~\ref{clusterthm}, and prove an analogue of Theorem~\ref{clusterthm}(a) for quaternion algebras. We discuss the theory of optimal embeddings of quadratic orders in Section~\ref{embeddingorders}, which enables us to prove a quaternion algebra analogue of Theorem~\ref{clusterthm}(b).  
These two results are translated into facts about supersingular curves in Section~\ref{msmallapplication}, where we finish the proof of Theorem~\ref{clusterthm}.

In Section~\ref{computations}, we discuss an algorithm that finds an isogeny between any two $M$-small supersingular curves. Appendix~\ref{appendixalgs} includes more detail on these algorithms, and give an example of its performance for $p\approx 2^{256}$ and $M=100$. We include bounds on the sizes of various sets of $M$-small curves in Appendix~\ref{appendixcount}. Appendix~\ref{diffprimeisog} depends on the results of Section~\ref{embeddingorders}, and shows that certain isogenies of degree $\ell$ cannot be replaced by short isogenies of degree relatively prime to $\ell$.

\paragraph{Acknowledgments.}
We would like to thank John Voight for fruitful discussion without which we would not have found the algorithms in Section \ref{computations}, and Akshay Venkatesh for pointing us towards the key ideas in Section~\ref{embeddingorders}. Thanks also to the anonymous reviewers for many improvements, including local proofs for Lemma~\ref{disteq} and Lemma~\ref{linkdown}, as well as pointing out a strengthening of Proposition~\ref{largeinters} that leads to a better bound in Theorem~\ref{clusterthm}.
	
\section{Generating $M$-small curves}\label{generatingmsmall}

Most well-known examples of supersingular curves are all $M$-small for relatively small values of $M$. For instance, supersingular curves with a non-trivial automorphism are $1$-small. This includes the curve $y^2=x^3+x$ with $j$-invariant $1728$ when $p\equiv 3\pmod 4$, and the curve $y^2=x^3+1$ with $j$-invariant $0$ when $p\equiv 2\pmod 3$.

More generally, Br\"{o}ker in~\cite{gencurves} proposes a general algorithm for producing a supersingular curve over an arbitrary finite field. We will discuss the algorithm here, and then see in Section \ref{modroots} how to generalize his approach to generate all $M$-small curves.

Given an imaginary quadratic field $K$, a \textbf{quadratic order} $\quadO$ in $K$ is a subring of $K$ such that the field of fractions of $\quadO$ is equal to $K$. If $\quadO_K$ is the ring of integers of $K$, the only quadratic orders in $K$ are of the form $\quadO_{K,f}:=\ZZ+f\quadO_K$ for some positive integer $f$, called the \textbf{conductor} of the quadratic order. If $D$ is the discriminant of $K$, then $d:=f^2D$ is the discriminant of $\quadO_{K,f}$ (throughout this paper, we will use $D$ to refer to fundamental discriminants, and $d$ to refer to discriminants of arbitrary quadratic orders). Further, any $d\equiv 0\text{ or }1\pmod 4$ can be written uniquely as $d=f^2D$ for $f>1$ and a fundamental discriminant $D$, so that quadratic orders are uniquely determined by their discriminant. We have $\quadO_{K,f}\subseteq \quadO_{K,g}$ if and only if $g\mid f$.

\begin{defn}
	Let $\quadO$ be a quadratic order. The \textbf{Hilbert class polynomial} $H_\quadO(x)\in\ZZ[x]$ is the monic irreducible polynomial characterized by the following property:\footnote{See for example~\cite[Proposition 13.2]{cox} for proof that such a polynomial exists.} for $j\in\CC$, $H_\quadO(j)=0$ if and only if $j$ is the $j$-invariant of an elliptic curve $\widetilde{E}$ over $\CC$ with $\End(\widetilde{E})\cong \quadO$.
\end{defn}

Br\"{o}ker's algorithm~\cite[Algorithm 2.4]{gencurves} proceeds as follows. To construct a supersingular curve over $\FF_p$ with $p\equiv 1\pmod 4$,\footnote{For $p=2$, the curve $y^2+y=x^3$ is supersingular, and for $p\equiv 3\pmod 4$ the curve $y^2=x^3+x$ is supersingular.} one first finds a prime $q\equiv 3\pmod 4$ with Legendre symbol $\legendre{-q}{p}=-1$. One can typically find very small values of $q$ satisfying these constraints. The algorithm proceeds by computing the Hilbert class polynomial $H_{\quadO_K}(x)$ for $K=\QQ(\sqrt{-q})$, and finding a root of $H_{\quadO_K}(x)\pmod p$ in $\FF_p$. The condition $\legendre{-q}{p}=-1$ then guarantees that this root is the $j$-invariant of a supersingular curve (Proposition \ref{msmallphi}). This algorithm generates $M$-small curves for a reasonably small value of $M$, as the following proposition shows.

\begin{prop}\label{brokerbound}
    The supersingular curves found by Algorithm 2.4 of~\cite{gencurves} are $\left(\frac{q+1}{4}\right)$-small. Assuming GRH, they are $M$-small for $M=O(\log^2 p)$.
\end{prop}
\begin{proof} 
    The output of the algorithm is a curve $E$ over $\FF_p$ with the following property: there exists a curve $\widetilde{E}$ over the Hilbert class field of $\QQ(\sqrt{-q})$ such that $\End(\widetilde{E})\cong\quadO_K$ and $E$ is the reduction of $\widetilde{E}$ modulo some prime of $\quadO_L$. In particular, $\frac{1+\sqrt{-q}}{2}\in\quadO_K$ is a non-integer endomorphism of $\widetilde{E}$ with norm $\frac{q+1}{4}$. The reduction map $\End(\widetilde{E})\to\End(E)$ is a degree-preserving injection~\cite[Proposition II.4.4]{silverman_adv}, so $\End(E)$ also contains a non-integer endomorphism of norm $\frac{q+1}{4}$, proving that $E$ is $\left(\frac{q+1}{4}\right)$-small.
    
    As discussed in the proof of Lemma 2.5 in~\cite{gencurves}, under GRH we can find $q=O(\log^2 p)$ with the desired properties.
\end{proof}

\subsection{Classification of $M$-small curves}\label{modroots}

A suitable generalization of Br\"{o}ker's algorithm~\cite{gencurves} can be used to generate the set of all $M$-small curves. Instead of only considering roots of the Hilbert class polynomial $H_{\quadO_K}(x)$, we will consider the set of roots of $H_\quadO(x)\pmod p$ for all quadratic orders with discriminant $-4M\leq \disc \quadO<0$. This set is the set of $j$-invariants of $M$-small curves (Proposition \ref{msmallphi}), and we can determine whether a $j$-invariant is supersingular or ordinary by a Legendre symbol calculation as in~\cite{gencurves}.

Sutherland gives an algorithm for computing $H_\quadO(x)\pmod p$ in time $O(|\disc \quadO|^{1+\varepsilon})$ \cite[Theorem 1]{suth_classpoly}. Computing $H_\quadO(x)$ for all quadratic orders of discriminant $-4M\leq d<0$ can therefore be done in time $O(M^{2+\varepsilon})$.

\begin{prop}\label{msmallphi}
	Let $3\leq M < p$, let $E$ be an elliptic curve over a finite field of characteristic $p$, and let $j$ be the $j$-invariant of $E$. Then $E$ is $M$-small if and only if $H_\quadO(j)=0\pmod p$ for some quadratic order $\quadO$ with discriminant $-4M\leq\disc \quadO<0$. Further, $E$ is supersingular if and only if $p$ does not split in the field of fractions of $\quadO$.
\end{prop}

\begin{proof}
    First suppose $E$ is $M$-small, and take $\alpha\in\End(E)-\ZZ$ for which $\deg\alpha\leq M$. By Deuring's Lifting Theorem~\cite[Theorem 13.14]{lang}, there is an elliptic curve $\widetilde{E}$ defined over a number field $L$, an endomorphism $\widetilde{\alpha}$ of $\widetilde{E}$, and a prime $\mathfrak p$ of $L$, such that $\widetilde{E}$ has good reduction at $\mathfrak p$, the reduction of $\widetilde{E}$ at $\mathfrak p$ is isomorphic over $\overline{\FF_p}$ to $E$, and that the endomorphism on $E$ induced by $\widetilde{\alpha}$ is equal to $\alpha$. Since the map $\End(\widetilde{E})\to\End(E)$ induced by reduction preserves degree~\cite[Proposition II.4.4]{silverman_adv}, $\widetilde{\alpha}\in \End(\widetilde{E})-\ZZ$ has degree at most $M$. For some quadratic order $\quadO$ in an imaginary quadratic field $K$, we will have $\End(\widetilde{E})\cong\quadO$ ~\cite[Corollary III.9.4]{silverman}. Letting $d=\disc\quadO$, we will have $\widetilde{\alpha}=\frac{a+b\sqrt{d}}{2}$ for some $a,b\in\ZZ$ with $b\neq 0$. Then
    \[\frac{|d|}{4}=\N_{K/\QQ}\left(\frac{\sqrt{d}}{2}\right)\leq\N_{\quadO/\ZZ}\left(\frac{a+b\sqrt{d}}{2}\right)=\deg\widetilde{\alpha}\leq M,\]
    implying $-4M\leq \disc \quadO<0$. By definition of the Hilbert class polynomial, this implies that the $j$-invariant $\widetilde{j}\in L$ of $\widetilde{E}$ is a root of the Hilbert class polynomial $H_\quadO(x)\in\ZZ[x]$. Reducing modulo $\mathfrak p$, we see that $j$ is a root of $H_\quadO(x)\pmod p$.
    
    Conversely, suppose $H_\quadO(j)=0\pmod p$ for some quadratic order $\quadO$ with discriminant $-4M\leq\disc \quadO<0$. Let $L/\QQ$ be the splitting field of $H_\quadO(x)$, and let $\mathfrak p$ be a prime over $p$ in $L$. 
    Then by considering the reductions mod $\mathfrak p$ of the linear factors of $H_\quadO(x)$, we can conclude that $j$ is the reduction mod $\mathfrak p$ of some $\widetilde{j}\in L$ with $H_\quadO(\widetilde{j})=0$. 
    If $\widetilde{E}$ is an elliptic curve over $L$ with $j$-invariant $\widetilde{j}$, then $\End(\widetilde{E})\cong\quadO$, and its reduction modulo $\mathfrak p$ is isomorphic over $\overline{\FF_p}$ to $E$. 
    If $d=\disc\quadO$ is congruent to $0\pmod 4$, then the element $\widetilde{\alpha}:=\frac{\sqrt{d}}{2}\in\quadO$ satisfies $\N_{\quadO/\ZZ}(\widetilde{\alpha})=\frac{|d|}{4}\leq M$. 
    If $d\equiv 1\pmod 4$, then we have $-4M+1\leq d$, and the element $\widetilde{\alpha}:=\frac{1+\sqrt{d}}{2}\in\quadO$ satisfies $\N_{\quadO/\ZZ}(\widetilde{\alpha})=\frac{|d|+1}{4}\leq M$. 
    Since the map $\End(\widetilde{E})\to\End(E)$ induced by reduction is a degree-preserving injection~\cite[Proposition II.4.4]{silverman_adv}, the reduction of $\widetilde{\alpha}$ in either case gives $\alpha\in \End(E)-\ZZ$ with $\deg\alpha\leq M$, so that $E$ is $M$-small.
    
    The fact that $E$ is supersingular if and only if $p$ does not split in the field of fractions of $\quadO$ is a theorem of Deuring~\cite[Theorem 13.12]{lang}.
\end{proof}
	
\section{Maximal Orders of Quaternion Algebras}\label{quatbackground}

In order to prove further results about $M$-small curves which are supersingular, we will need to review the theory of quaternion algebras. Unless otherwise cited, all the material in this section can be found in~\cite{voight_book}.

\subsection{Quaternion Algebras and Subfields}

There is a quaternion algebra $B$ over $\QQ$, unique up to isomorphism, that ramifies exactly at $p$ and $\infty$. For $p\neq 2$, we can take 
	\[\QQ\langle i,j,k\rangle:=\{w+xi+yj+zk:i^2=-q, j^2=-p, ij=-ji=k\}\]
for an appropriate integer $q$ depending on $p\pmod 8$ (for $p\equiv 1\pmod 8$, this agrees with the value of $q$ ~\cite[Proposition 5.1]{pizer}.
	
{\flushleft Given $\alpha=w+xi+yj+zk\in B$, we define:}
\begin{itemize}
	\item its \textbf{conjugate}, $\overline{\alpha}:=w-ix-jy-kz$. This satisfies the property that $\overline{\overline{\alpha}}=\alpha$, $\overline{\alpha+\beta}=\overline{\alpha}+\overline{\beta}$, and $\overline{\alpha\beta}=\overline{\beta}\overline{\alpha}$ for all $\alpha,\beta\in B$.
	\item its \textbf{reduced norm}, $\nrd(\alpha):=\alpha\overline{\alpha}=w^2+qx^2+py^2+qpz^2$.
	\item its \textbf{reduced trace}, $\trd(\alpha):=\alpha+\overline{\alpha}=2w$.
\end{itemize}
From these definitions, we see that any $\alpha\in B$ is the root of a polynomial
\[x^2-\trd(\alpha)x+\nrd(\alpha)\]
with rational coefficients; if $\alpha\notin\QQ$ this is the \textbf{minimal polynomial} of $\alpha$. Noting that $\trd(\alpha)^2-4\nrd(\alpha)<0$, any $\alpha\notin\QQ$ generates an imaginary quadratic subfield $\QQ(\alpha)\subseteq B$.  The following result (a consequence of the Skolem-Noether Theorem) tells us exactly when two elements of $B$ have the same minimal polynomial.

\begin{thm}[{\cite[Corollary 7.7.3]{voight_book}}]\label{skolnoeth}
	Let $\alpha,\beta\in B-\QQ$. Then $\alpha$ and $\beta$ satisfy the same minimal polynomial if and only if there exists $\gamma\in B^\times$ such that $\gamma^{-1}\alpha\gamma=\beta$.
\end{thm}

\noindent
In particular, given any two isomorphic quadratic subfields of $B$, applying this theorem to the generators shows that there is an automorphism of $B$ that takes one subfield onto the other. An imaginary quadratic field $K$ embeds into $B$ if and only if $p$ does not split in $K$~\cite[Proposition 14.6.7]{voight_book}, which is equivalent to requiring that the Legendre symbol $\left(\frac{D}{p}\right)$ is not equal to $1$, where $D$ is the discriminant of $K$.

\subsection{Ideals and Orders}

An \textbf{ideal} $I\subseteq B$ is a subgroup under addition which is generated by a basis of $B$ considered as a vector space over $\QQ$. An \textbf{order} $\QuatO\subseteq B$ is an ideal which contains $1$ and is closed under multiplication (and is hence a subring of $B$). An element $\alpha\in B$ with $\trd(\alpha),\nrd(\alpha)\in\ZZ$ is called \textbf{integral}; $\alpha$ is integral if and only if it is contained in some order of $B$.

Given an ideal $I\subseteq B$, we can define \textbf{left and right orders of $I$},
\[\QuatO_L(I):=\{x\in B:xI\subseteq I\},\qquad \QuatO_R(I):=\{x\in B:Ix\subseteq I\}.\]
We say that $I$ is a \textbf{left ideal of $\QuatO$} if $\QuatO_L(I)=\QuatO$, and that $I$ is a \textbf{right ideal of $\QuatO'$} if $\QuatO_R(I)=\QuatO'$. In this scenario we say $I$ \textbf{links $\QuatO$ to $\QuatO'$}.

An ideal $I$ that is closed under multiplication is called an \textbf{integral ideal}. An integral ideal is necessarily contained in its left and right orders, and hence $\nrd(\alpha)\in\ZZ$ for all $\alpha$ in an integral ideal. Given an integral ideal $I\subseteq B$, the \textbf{reduced norm} of $I$ is defined to be
\[\nrd(I):=\gcd\{\nrd(\alpha)\mid\alpha\in I\}.\]
Observe that $I\subseteq J$ implies $\nrd(J) \mid \nrd(I)$.

An order is \textbf{maximal} if there are no orders properly containing it. Unlike number fields, for which the ring of integers is the unique maximal order, a quaternion algebra will typically have many distinct maximal orders. 

Given a quadratic order $\quadO$ and a maximal order $\QuatO\subseteq B$ we say that $\quadO$ is \textbf{optimally embedded} in $\QuatO$ if $\quadO\cong \QuatO\cap K$ for some subfield $K\subseteq B$. 

\subsection{The Deuring Correspondence}\label{deuringsec}

(See Chapter 42 of~\cite{voight_book} for details.)

Let $\Sp\subseteq\FF_{p^2}$ denote the set of $j$-invariants of supersingular curves. Given $j\in \Sp$, $\End(E_j)$ will be isomorphic to a maximal order in $B$. If $j$ and $j^p$ are $\FF_{p^2}$-conjugates, then $\End(E_j)$ and $\End(E_{j^p})$ will be isomorphic orders. Aside from this relation, non-isomorphic curves will always have non-isomorphic endomorphism rings. In fact, we have a bijection, known as the \emph{Deuring correspondence}:
\[\Sp/(j\sim j^p)\leftrightarrow \{\text{maximal orders of }B\}/\cong\]
sending $j$ to the endomorphism ring of $E_j$. The degree (resp. trace, resp. dual) of an endomorphism is equal to the norm (resp. trace, resp. conjugate) of the corresponding element of $B$, and composition of endomorphisms corresponds to multiplication of elements of $B$. Further, suppose we fix a maximal order $\QuatO_j$ associated to $\End(E_j)$ for some $j$. Then we have a one-to-one correspondence
\begin{align*}
\{\text{separable isogenies out of }E_j\}/\cong &\leftrightarrow\{\text{left ideals of }\QuatO_j\}.
\end{align*}
An isogeny $\phi:E_j\to E'$ will correspond to an ideal $I$ linking $\QuatO_j$ to some maximal order $\QuatO_{j'}$ isomorphic to $\End(E')$ (that is, $I$ is a left $\QuatO_j$-ideal and a right $\QuatO_{j'}$-ideal), and $\deg\phi=\nrd(I)$. 
	\section{Distance Between Maximal Orders}\label{quatproofs}

\subsection{Definitions for Maximal Orders}

In order to use the Deuring correspondence to express Theorem~\ref{clusterthm} in the language of maximal orders, we must have a notion of $M$-small and a notion of distance for maximal orders. The first of these is straightforward.

\begin{defn}
    An order $\QuatO\subseteq B$ is \textbf{$M$-small} if there exists $\alpha\in\QuatO-\ZZ$ with $\nrd(\alpha)\leq M$.
\end{defn}

\noindent
Then a supersingular curve is $M$-small if and only if its endomorphism ring is an $M$-small maximal order. Our next task is to come up with a definition of distance between maximal orders that is compatible with Definition~\ref{isogdist}.

\begin{lem}\label{disteq}
    If $\QuatO,\QuatO'\subseteq B$ are maximal orders, the following quantities are all equal:
    \begin{enumerate}[label=(\alph*)]
        \item $|\QuatO:\QuatO\cap\QuatO'|$ (the index of $\QuatO\cap\QuatO'$ in $\QuatO$).
        \item $|\QuatO':\QuatO\cap\QuatO'|$ (the index of $\QuatO\cap\QuatO'$ in $\QuatO'$).
        \item The smallest reduced norm of an integral ideal linking $\QuatO$ to $\QuatO'$. 
    \end{enumerate}
\end{lem}
\begin{proof}
    We observe that these values are equal if and only if the corresponding quantities obtained by localizing at each prime are all equal~\cite[Lemma 9.5.7]{voight_book}. There is a unique maximal order at the ramified prime $p$, and so all three of the local quantities at $p$ are equal to $1$. 
    
    For $\ell\neq p$, the statement follows from  the theory of the Bruhat-Tits Tree~\cite[Section 23.5]{voight_book}. Specifically, we have $B_\ell\cong M_2(\QQ_\ell)$. With respect to an appropriate basis, if we set $\varpi=\begin{psmallmatrix}\ell&0\\0&1\end{psmallmatrix}$, we will have $\QuatO_\ell=M_2(\ZZ_\ell)$ and $\QuatO'_\ell=\varpi^{-e}\QuatO_\ell\varpi^{e}$ for some exponent $e$~\cite[Lemma 23.5.14]{voight_book}. Then $\QuatO_\ell\varpi^e=\varpi^e\QuatO'_\ell$ is the linking ideal of smallest reduced norm, and we can check directly that
    \[|\QuatO_\ell:\QuatO_\ell\cap\QuatO'_\ell|=|\QuatO'_\ell:\QuatO_\ell\cap\QuatO'_\ell|=\nrd(\QuatO_\ell\varpi^e)=\ell^e.\qedhere\]
\end{proof}

\begin{defn}
    The \textbf{distance from $\QuatO$ to $\QuatO'$}, $d(\QuatO,\QuatO')$, is any of the equivalent quantities in Lemma~\ref{disteq}.
\end{defn}

\noindent
Note that $\log d$ defines a metric on the set of maximal orders of $B$. Positive-definiteness and symmetry follow immediately from definition. If $I$ and $J$ are the integral ideals of smallest reduced norm linking $\QuatO$ to $\QuatO'$ and $\QuatO'$ to $\QuatO''$, respectively, then $IJ$ is an integral ideal linking $\QuatO$ to $\QuatO''$. Since $\nrd(IJ)\leq \nrd(I)\nrd(J)$ for any compatible ideals $I$ and $J$~\cite[Example 16.3.6]{voight_book}, $\log d$ satisfies the triangle inequality. We can compare distances between elliptic curves and distances between maximal orders as follows.
\begin{lem}\label{distcompare}
	Let $E$ and $E'$ be supersingular curves. Then
	\[d(E,E')=\min\{d(\QuatO,\QuatO')\mid \QuatO\cong \End(E),\QuatO'\cong\End(E')\}.\]
\end{lem}
\begin{proof}
	By the Deuring correspondence, both sides are equal to
	\[\min\{\deg\phi\mid \phi:E\to E''\text{ for some $E''$ with }\End(E'')\cong\End(E')\}.\qedhere\]
\end{proof}

\subsection{Two Key Propositions}

Suppose that $\QuatO$ and $\QuatO'$ are each $M$-small maximal orders in $B$. Let $\alpha\in \QuatO-\ZZ$ and $\alpha'\in\QuatO'-\ZZ$ each have reduced norm at most $M$. We will show that the distance from $\QuatO$ to $\QuatO'$ is small if $\QQ(\alpha)$ is isomorphic to $\QQ(\alpha')$, and is large otherwise. Precisely, we will prove the following two results, which are quaternion algebra analogues of results (a) and (b) of Theorem~\ref{clusterthm}.

\begin{prop}\label{largeinters}
	If $\QQ(\alpha)\not\cong\QQ(\alpha')$, then $d(\QuatO,\QuatO')^2\geq\frac{p}{4M^2}$.
\end{prop}

\begin{prop}\label{opath}
	If $\QQ(\alpha)\cong\QQ(\alpha')$, then there exists a sequence of (not necessarily distinct) maximal orders 
	\[\QuatO=\QuatO_0,\QuatO_1,\ldots,\QuatO_r\cong\QuatO'\]
	such that
	\begin{itemize}
		\item the distance between two consecutive terms is at most $\frac{4}{\pi}\sqrt{M}$, and
		\item each $\QuatO_i$ contains an element with the same minimal polynomial as either $\alpha$ or $\alpha'$.
	\end{itemize}  
	We can find such a sequence with $r\leq 3$, or alternatively we can find such a sequence such that consecutive orders are linked by an ideal of prime norm at most $\frac{4}{\pi}\sqrt{M}$.
\end{prop}

Proposition~\ref{opath} will be proven in Section~\ref{embeddingorders}; we will proceed with a proof of Proposition~\ref{largeinters}. We begin by quoting a theorem due to Kaneko:

\begin{thm}\cite[Theorem $2'$]{kaneko}\label{subfieldbound}
	Let $\QuatO\subseteq B$ be a maximal order. If $\quadO$ and $\quadO'$ are quadratic orders of imaginary quadratic fields, optimally embedded into $\QuatO$ with distinct images, then $\disc \quadO\disc \quadO'\geq 4p$. If in addition $\quadO$ and $\quadO'$ have isomorphic fields of fractions, then $\disc \quadO\disc \quadO'\geq p^2$.
\end{thm}

\noindent
The proof proceeds by explicitly computing the discriminant of the suborder generated by $\quadO$ and $\quadO'$; noting that it must be a multiple of $p^2$ gives the desired inequality. Using this, we can prove our first bound.

\begin{proof}[Proof of Proposition~\ref{largeinters}]
    Let 
    \[\quadO:=\QQ(\alpha)\cap\QuatO\qquad\text{and}\qquad \quadO':=\QQ(\alpha')\cap\QuatO.\]
    both be optimally embedded in $\QuatO$. Since $\QQ(\alpha)\not\cong\QQ(\alpha')$, these are distinct, so Theorem~\ref{subfieldbound} implies that $\disc \quadO\disc \quadO'\geq 4p$. 
    
    Let $D$ denote the discriminant of $K=\QQ(\alpha)$. Since $\alpha\in \quadO-\ZZ$, and the quadratic order $\quadO$ must be of the form $\quadO=\ZZ+f\quadO_K$ for some positive integer $f$, we have
    \[\nrd(\alpha)\geq \N_{K/\QQ}\left(\tfrac12 f\sqrt{D}\right)=\frac{f^2D}{4}=\frac14 \disc \quadO.\]
    
    Letting $d=d(\QuatO,\QuatO')=|\QuatO':\QuatO\cap\QuatO'|$, we have $d\alpha'\in \QuatO\cap\QuatO'\subseteq\QuatO$. As above, we can compute $d^2\nrd(\alpha')\geq\frac 14\disc \quadO'$. Hence
    \[d^2\geq \frac{\disc\quadO'}{4\nrd(\alpha')}\frac{\disc\quadO}{4\nrd(\alpha)}\geq \frac{p}{4M^2}. \qedhere\]
\end{proof}

\section{Optimal Embeddings}\label{embeddingorders}

Let $K$ be an imaginary quadratic field of discriminant $D$, and let two maximal orders $\QuatO,\QuatO'$ of $B$ each admit an optimal embedding of some quadratic order of $K$. If these optimally embedded quadratic orders both have small discriminant, our goal is to construct a sequence of maximal orders from $\QuatO$ to $\QuatO'$ such that the distance between two consecutive orders is small. 

To do this, we will need to consider two types of relations between maximal orders. If two maximal orders have the same quadratic order optimally embedded in each, we call the relationship between them a ``horizontal step''; if one of the optimally embedded orders is a proper subset of the other, the relationship is called a ``vertical step''.\footnote{The terminology is meant to draw a comparison with isogeny graphs of ordinary elliptic curves, in which there are horizontal isogenies which preserve the endomorphism ring and vertical isogenies which change it~\cite{suth_isogvolc}.} 

\begin{rmk}
    A fixed embedding $K\hookrightarrow B$ defines a unique optimally embedded quadratic order $K\cap \QuatO$. However, there is not a unique embedding of $K$ into $B$ (see Theorem~\ref{skolnoeth}), so it is possible for multiple distinct quadratic orders of $K$ to all optimally embed into a single maximal order. Strictly speaking, ``horizontal'' and ``vertical'' steps are not relations between maximal orders, but rather between pairs of the form $(\QuatO,\quadO)$, where $\quadO$ is a quadratic order optimally embedded in a maximal order $\QuatO$; if we do not specify $\quadO$, then these relations will not be well-defined.
\end{rmk}

\subsection{Horizontal Steps}

First we consider the case in which the same quadratic order $\quadO$ is optimally embedded in two maximal orders $\QuatO$ and $\QuatO'$. For this we will use a version of the Chevalley-Hasse-Noether Theorem proved by Eichler.%

\begin{thm}[{\cite[Satz 7]{eichler}}]\label{chevhassnoeth}
    Let $\QuatO,\QuatO'\subseteq B$ be two maximal orders, and suppose
	\[\quadO\cong K\cap\QuatO=K\cap\QuatO'\]
	is optimally embedded in each. Then there is an invertible ideal\footnote{Eichler simply states that there must exist an ideal $\mathfrak{a}$ of $\quadO$ with $\QuatO\mathfrak{a}=\mathfrak{a}\QuatO'$, but he defines ideals to be locally principal~\cite[133]{eichler} and this implies invertibility.} $\mathfrak{a}$ of $\quadO$ such that $\QuatO\mathfrak{a}=\mathfrak{a}\QuatO'$.
\end{thm}

Using this theorem, we will show that the distance between orders related by a horizontal step can be bounded in terms of norms of ideals of the common optimally embedded order $\quadO$. Recall from Section~\ref{generatingmsmall} that we set $\quadO_{K,f}:=\ZZ+f\quadO_K$, the order of conductor $f$ in $\quadO_K$.

\begin{lem}\label{linkacross}
	Let $K$ be an imaginary quadratic field of discriminant $D$. Let $\QuatO,\QuatO'\subseteq B$ be maximal orders, and suppose $\quadO:=\quadO_{K,f}$ optimally embeds into each. Then there exists an automorphism $\phi:B\to B$ such that
	\[d(\QuatO,\phi(\QuatO'))\leq \frac{2}{\pi}f\sqrt{|D|}.\]
\end{lem}

\begin{proof}
	By the Skolem-Noether theorem (Theorem~\ref{skolnoeth}), there exists $\gamma\in B$ such that 
	\[\QuatO\cap K=(\gamma^{-1}\QuatO'\gamma)\cap K\]
	for some embedding $K\hookrightarrow B$.
	Since conjugation by $\gamma$ is an automorphism of $B$, we can replace $\QuatO'$ with $\gamma^{-1}\QuatO'\gamma$, so that $\QuatO\cap K=\QuatO'\cap K$ is identified with the quadratic order $\quadO$.
	
	By Theorem~\ref{chevhassnoeth}, there exists an invertible ideal $\mathfrak{a}$ of $\quadO$ such that $\QuatO\mathfrak{a}=\mathfrak{a}\QuatO'$. We can find an ideal $\mathfrak{b}$ in the same ideal class as $\mathfrak{a}$ (so $\mathfrak{b}=\mathfrak{a}\delta$ for some $\delta\in K$) with $N_{\quadO/\ZZ}(\mathfrak{b})\leq\frac{2}{\pi}f\sqrt{|D|}$ by Minkowski's bound~\cite[Theorem 5.4]{stevenhagen}. Then
	\[\QuatO\mathfrak{b}=\QuatO\mathfrak{a}\delta=\mathfrak{a}\QuatO'\delta=(\mathfrak{b}\delta^{-1})\QuatO'\delta=\mathfrak{b}\phi(\QuatO'),\]
	where $\phi:B\to B$ is the automorphism $\phi(x):=\delta^{-1}x\delta$. Hence $\QuatO\mathfrak{b}=\mathfrak{b}\phi(\QuatO')$ is an ideal linking $\QuatO$ to $\phi(\QuatO')$. We have
	\[\nrd(\QuatO\mathfrak{b})=\gcd\{\nrd(x)\mid x\in\QuatO\mathfrak{b}\}\leq\gcd\{\N_{\quadO/\ZZ}(x)\mid x\in \mathfrak{b}\}=\N_{\quadO/\ZZ}(\mathfrak{b})\leq\frac{2}{\pi}f\sqrt{|D|},\]
	which gives the upper bound on $d(\QuatO,\phi(\QuatO'))$. 
\end{proof}

\subsection{Vertical Steps}

Now we must determine how to step between maximal orders that have different quadratic orders optimally embedded into each. If a quadratic order $\quadO\neq\quadO_K$ optimally embeds into a maximal order $\QuatO$, the following lemma explicitly constructs a new maximal order with an optimally embedded quadratic order of smaller conductor.

\begin{lem}\label{linkdown}
	Let $\ell$ be a prime, and $\beta\in\quadO_K$. Let $\QuatO\subseteq B$ be a maximal order in which $\ZZ[\ell\beta]$ optimally embeds. Then there exists a maximal order $\QuatO'$ in which $\ZZ[\beta]$ optimally embeds, with $d(\QuatO,\QuatO')=\ell$.
\end{lem}

\begin{proof}
    Consider $\QuatO_\ell\subseteq B_\ell$, given by completing at $\ell$. By~\cite[Proposition 30.5.3]{voight_book}, there are no optimal embeddings of $\ZZ[p\beta]$ in $\QuatO_p$ (so the conditions of the lemma cannot be satisfied if $\ell=p$), and for $\ell\neq p$ there is a unique optimal embedding of $\ZZ[\ell\beta]$ in $\QuatO_\ell$ up to conjugation. Explicitly, for $\ell\neq p$ we will have $\QuatO_\ell\cong M_2(\ZZ_p)$~\cite[Corollary 10.5.5]{voight_book}, and if $\beta^2-t\beta+n=0$ is the minimal polynomial for $\beta$, then the embedding $K\to M_2(\QQ_p)$ defined by
    \[\ell\beta\mapsto\begin{pmatrix}
    0&-\ell^2n\\1&\ell t
    \end{pmatrix},\]
    induces an optimal embedding of $\ZZ[\ell\beta]$ into $M_2(\ZZ_p)$, unique up to conjugation by $\GL_2(\ZZ_p)$. The maximal order
    \[\QuatO'_\ell:=\begin{pmatrix}
    \ZZ_\ell &\ell\ZZ_\ell \\ \ell^{-1}\ZZ_\ell & \ZZ_\ell
    \end{pmatrix}\subseteq M_2(\QQ_p)\]
    contains $\begin{psmallmatrix} 0 & -\ell n\\ 1/\ell & t
    \end{psmallmatrix}$, the image of $\beta$, but does not contain the image of $\frac1\ell \beta$.
    
    Now let 
    \[\QuatO':=\QuatO'_\ell\cap\bigcap_{q\neq \ell}\QuatO_q.\]
    This is a maximal order because it is maximal at every prime. For all $q\neq\ell$ we have $\ell^{-1}\in\ZZ_q$, so $\quadO_{K, f}$ embeds into $\QuatO_q$, and hence $\quadO_{K, f}$ optimally embeds into $\QuatO'$. Finally, since $\QuatO$ and $\QuatO'$ are equal at every prime besides $\ell$, we have 
    \[d(\QuatO,\QuatO')=|\QuatO:\QuatO\cap\QuatO'|=|\QuatO_\ell:\QuatO_\ell\cap\QuatO'_\ell|=\ell.\qedhere\]
\end{proof}

\begin{cor}\label{bigstepdown}
    Let $f$ be a positive integer, and $\QuatO\subseteq B$ be a maximal order in which $\quadO_{K,f}$ optimally embeds. Then there exists a maximal order $\widetilde{\QuatO}$ in which $\quadO_K$ optimally embeds, with $d(\QuatO,\widetilde{\QuatO})\leq f$.
\end{cor}
\begin{proof}
    Factor $f=\ell_1\cdots\ell_k$ into primes, and set $f_i=\ell_{i+1}\cdots\ell_k$ (so $f_0=f$ and $f_k=1$). Apply Lemma~\ref{linkdown} successively, obtaining maximal orders $\QuatO=:\QuatO_0,\QuatO_1,\ldots,\QuatO_k=\widetilde{\QuatO}$, where $\quadO_{K,f_i}$ optimally embeds in $\QuatO_i$. Then
    \[d(\QuatO,\widetilde{\QuatO})\leq \prod_{i=1}^k d(\QuatO_{i-1},\QuatO_i)=\prod_{i=1}^k \ell_i=f.\qedhere\]
\end{proof}

\subsection{Proof of Proposition~\ref{opath}}
		We are now ready to combine our vertical and horizontal steps to create a path between two maximal orders $\QuatO$ and $\QuatO'$. Let $K\cong\QQ(\alpha)$. To begin, take a vertical step from each of $\QuatO$ and $\QuatO'$ using Corollary~\ref{bigstepdown}: we obtain maximal orders $\widetilde{\QuatO}$ and $\widetilde{\QuatO'}$, both containing an optimally embedded $\quadO_K$, as well as bounds on $d(\QuatO,\widetilde{\QuatO})$ and $d(\QuatO',\widetilde{\QuatO'})$. Now join $\widetilde{\QuatO}$ and $\widetilde{\QuatO'}$ by a horizontal step using Lemma~\ref{linkacross}: this gives us an automorphism $\phi:B\to B$ and a bound on $d(\widetilde{\QuatO},\phi(\widetilde{\QuatO}'))$. Combining these steps, we obtain a sequence 
		\[\QuatO,\,\widetilde{\QuatO},\,\phi(\widetilde{\QuatO}'),\,\phi(\QuatO')\cong \QuatO'\]
		with bounds on the consecutive distances; we can check that these are all bounded above by $\frac{4}{\pi}\sqrt{M}$. Since each of these orders contains an element with the same minimal polynomial as $\alpha$ or $\alpha'$, this settles the $r\leq 3$ case of the Proposition.
		
		If instead we want all consecutive terms to be linked by ideals of prime norm, we can break up each step into smaller ones. For the vertical steps, we can factor the conductor of the optimally embedded orders into primes and take one step for each prime, as in the proof of Corollary~\ref{bigstepdown}. For the horizontal step, We can factor $\mathfrak{b}$ (from the proof of Lemma~\ref{linkacross}) into prime ideals as $\mathfrak{p}_1\cdots\mathfrak{p}_s$. Set $\QuatO_0=\QuatO$, and for each $i=1,\ldots,s$, recursively define \[\QuatO_i:=\mathfrak{p}_i^{-1}\QuatO_{i-1}\mathfrak{p}_i.\] Then $\quadO_K$ is optimally embedded in each $\QuatO_i$, and consecutive orders $\QuatO_{i-1}$ and $\QuatO_i$ are linked by the ideal $\QuatO_{i-1}\mathfrak{p}_i=\mathfrak{p}_i\QuatO_i$ of norm $\N_{K/\QQ}(\mathfrak{p}_i)$. If we assume $\mathfrak{b}$ was chosen to be minimal, none of the $\mathfrak{p}_i$ can be principal, and so they will all have prime norm.\hfill\qed

	\section{Proof of Theorem~\ref{clusterthm}}\label{msmallapplication}

\subsection{Existence of Partition}\label{partitionexist}
	
	Recall that we defined $\Sp\subseteq\FF_{p^2}$ to be the set of all $j$-invariants of supersingular curves. For each fundamental discriminant $-4M\leq D<0$ which is not congruent to a square mod $p$ (that is, for which the Legendre symbol $\left(\frac{D}{p}\right)$ is equal to $-1$), set
	\[T_D:=\{j\in \Sp:\QQ(\alpha)\cong\QQ(\sqrt{D})\text{ for some }\alpha\in\End(E_j)-\ZZ,\deg\alpha\leq M\}.\]
	We must prove that the sets $T_D$ are disjoint, nonempty, and that every $j\in \SpM$ is in some $T_D$.
	
	If $j\in T_D\cap T_{D'}$, then $\End(E_j)$ contains elements $\alpha,\alpha'$ generating nonisomorphic subfields, each with degree at most $M$. Applying Proposition~\ref{largeinters} to $\QuatO=\QuatO'\cong\End(E_j)$, we obtain $1=d(\QuatO,\QuatO)^2\geq\frac{p}{4M^2}$, contradicting $p>4M^2$. Hence the sets $T_D$ are all disjoint.
	
	For any $-4M\leq D<0$ with $\left(\frac{D}{p}\right)= -1$, either $\alpha=\frac{\sqrt{D}}{2}$ or $\alpha=\frac{1+\sqrt{D}}{2}$ is an integral element of $\alpha\in\QQ(\sqrt{D})-\QQ$, and the norm of $\alpha$ will be respectively $\frac{-D}{4}\leq M$ or $\frac{1-D}{4}\leq M$ (since it must be an integer).
	By the constraints on $D$, $\QQ(\alpha)$ embeds into $B$~\cite[Proposition 14.6.7]{voight_book}, and so the integral element $\alpha$ is contained in some maximal order. By the Deuring correspondence, this order is isomorphic to $\End(E_j)$ for some $j\in \Sp$. Hence there is an embedding $\iota:\ZZ[\alpha]\to \End(E_j)$, so $j$ is $M$-small. Since $\QQ(\iota(\alpha))\cong\QQ(\sqrt{D})$, we have $j\in T_D$, and so $T_D$ is nonempty.
	
	Suppose $j\in \SpM$, so there exists $\alpha\in\End(E_j)-\ZZ$ with $\deg(\alpha)\leq M$. Taking the minimal polynomial $x^2-tx+\deg(\alpha)$ of $\alpha$, $\deg(\alpha)\leq M$ implies $-4M\leq t^2-4\deg(\alpha)<0$. Dividing by perfect square factors does not affect these inequalities, and so the discriminant $D$ of $\QQ(\alpha)$ must be in the range $-4M\leq D<0$. Since $\QQ(\alpha)$ embeds into $B$ and $D<p$, we must have $\left(\frac{D}{p}\right)=-1$. Hence $j$ is in $T_D$ for some $D$.
	
\subsection{Distance between $T_D$ and $T_{D'}$}
	
	Suppose $j\in T_D$ and $j'\in T_{D'}$ for $D\neq D'$. For any $\QuatO\cong \End(E_j)$ and $\QuatO'\cong \End(E_{j'})$, we have $d(\QuatO,\QuatO')^2\geq\frac{p}{4M^2}$ by Proposition~\ref{largeinters}. Thus Lemma~\ref{distcompare} tells us that 
	\[d(j,j')=\min\{d(\QuatO, \QuatO')\mid \QuatO\cong\End(E_j),\,\QuatO'\cong\End(E_{j'})\}\geq\frac{\sqrt{p}}{2M}.\]
	
\subsection{Distances within $T_D$}
	
	Suppose $j,j'\in T_D$, and let $\alpha,\alpha'$ be the corresponding small non-integer endomorphisms with $\QQ(\alpha)\cong\QQ(\alpha')\cong\QQ(\sqrt{D})$. By Proposition~\ref{opath}, there exists a chain 
	\[\End(E_j)\cong\QuatO_0,\QuatO_1,\ldots,\QuatO_r\cong \End(E_{j'})\]
	with consecutive distances bounded by  $\frac{4}{\pi}\sqrt{M}$, and each containing an element with the same minimal polynomial as either $\alpha$ or $\alpha'$. 
	
	Now set $j_0:=j$, $j_r:=j'$, and for each $i=1,\ldots,r-1$, set $j_i$ so that $\End(E_{j_i})\cong \QuatO_i$. By Lemma~\ref{distcompare}, for $i=1,\ldots,r$ we have
	\[d(E_{j_{i-1}},E_{j_i})\leq d(\QuatO_{i-1},\QuatO_i)\leq\frac{4}{\pi}\sqrt{M}.\]
	Because each $E_{j_i}$ has an element with the same minimal polynomial as $\alpha$ or $\alpha'$, each $j_i\in T_D$. This shows that the sequence $j_0,j_1,\ldots,j_r$ has the desired properties.
	
	Note that we could have chosen our sequence of maximal orders to have $r\leq 3$, or to have consecutive orders linked by an ideal of prime order. In the first case, we would have a sequence of $j$-invariants with $r\leq 3$. In the second case, an ideal linking $\QuatO_{i-1}$ to $\QuatO_i$ with prime norm at most $\frac{4}{\pi}\sqrt{M}$ will correspond by the Deuring correspondence to an isogeny $E_{j_{i-1}}\to E_{j_i}$ of prime degree at most $\frac{4}{\pi}\sqrt{M}$. This concludes the proof.\qed

\section{Isogenies Between $M$-small Supersingular Curves}\label{computations}

Despite the large distances between $M$-small curves in distinct subsets $T_D$ (as in Theorem~\ref{clusterthm}), we show that isogenies between them can nonetheless be computed efficiently (probabilistic polynomial time in $M$ and $\log p$) under certain heuristic assumptions. To begin with, we recall the following observations, made in other papers:
\begin{enumerate}[label=(O\roman*)]
    \item\label{obs1} Given two maximal orders $\QuatO$ and $\QuatO'$, an ideal linking $\QuatO$ to $\QuatO'$ with $S$-powersmooth norm ($S\approx \frac72\log p$) can be computed efficiently~\cite[Sections 4.5--4.7]{quatisogpath}.
    
    \item\label{obs2} Given a supersingular elliptic curve $E$ with known endomorphism ring $\End(E)$, and a left ideal of $\End(E)$ with $S$-powersmooth norm ($S\approx \frac72\log p$), an isogeny out of $E$ corresponding to $I$ under the Deuring correspondence can be computed efficiently~\cite[Proposition 4]{isoggraphs_endorings}.
	
	\item\label{obs3} Given a maximal order $\QuatO$, a $j$-invariant such that $\End(E_j)\cong\QuatO$ can be computed efficiently \cite[Section 7.1]{isoggraphs_endorings}.
\end{enumerate}

For each $T_D$, we can construct a maximal order $\QuatO_D$, and use Observation~\ref{obs3} to find a $j$-invariant $j_D\in T_D$ with known endomorphism ring. Then for $D\neq D'$, we can use Observations~\ref{obs1} and~\ref{obs2} to find a (large degree) isogeny from $j_D$ to either $j_{D'}$ or $j_{D'}^p$ as a composition of many small isogenies. These specified $j$-invariants $j_D$ will act as ``airports''; knowing that each isogeny valley $T_D$ is connected by small-degree isogenies, we can connect any two $M$-small supersingular curves by first finding a path from each to the closest airport, then following the large degree path between the airports. These algorithms are discussed further in Appendix~\ref{appendixalgs}.
	
\paragraph{Isogenies defined over $\FF_p$.}
Suppose $j_1$ and $j_2$ are $M$-small $j$-invariants in $\FF_p$. Some situations, such as key recovery for the CSIDH protocol~\cite{csidh}, require being able to find an $\FF_p$-isogeny $E_{j_1}\to E_{j_2}$. While our algorithm allows us to construct an isogeny between these curves, this isogeny will not necessarily be defined over $\FF_p$. This is solved by concurrent work of Castryck, Panny, and Vercauteren~\cite{ratfromirrat}, in which they provide an algorithm to compute an $\FF_p$-isogeny $E_{j_1}\to E_{j_2}$, given the endomorphism rings of $E_{j_1}$ and $E_{j_2}$ (which we are able to compute).

	\printbibliography
	
	\pagebreak
    \appendix
\section{Computing Isogenies Between $M$-small Supersingular Curves}\label{appendixalgs}

Let us elaborate on the approach described in Section~\ref{computations}.
One subtle issue with this method comes from the fact that the Deuring correspondence is not one-to-one; it's quite possible that for some $D$, $T_D$ is actually a disjoint union of two subsets that are very far apart, one being the set of conjugates of the other. To remedy this, it suffices to have a single $M$-small supersingular $j$-invariant $j_0\in\FF_p$ to route all paths through. For then if we have a path from $j_0$ to $j^p$, we can simply apply the $p^\text{th}$ power Frobenius map to this path to obtain a path from $j_0$ to $j$. This technique will be used in Algorithm \ref{findisog}.

\subsection{Assumptions}\label{assumptions}
Recall that $i^2=-q$ and $j^2=-p$ for some relatively small value of $q$. Let $K\neq \QQ(i)$ be a quadratic field of discriminant $-4M\leq D<0$. We will make two assumptions which are unproven but heuristically reasonable. In Section~\ref{exmple} we carry out computations that depend on these assumptions for $p\approx 2^{256}$, $M=100$, and all allowable values of $D$, showing that in practice these assumptions seem to be valid.
\begin{enumerate}[label=(A\roman*)]
    \item\label{normeq} A solution $z\in L:=K(i)$ to the norm equation $\N_{L/K}(z)=-p$ can be computed efficiently, if one exists.\footnote{The algorithm for doing so is described in~\cite[Section 6]{simon}, and is implemented in Magma~\cite{magma} as \texttt{NormEquation(L, -p)}. In general, the bottleneck of the algorithm used to solve $\N_{L/K}(z)=m$ is to factor $m$ into primes of $K$, but this is easy in our case because $p$ is already an integer prime.}
    
    \item\label{easilyfactor} Let $\omega\in B$ satisfy $4\omega^2=D$ (if $D\equiv 0\pmod 4$) or $4\omega^2-4\omega+1=D$ (if $D\equiv 1\pmod 4$). Then if we randomly select integral elements $\beta\in B$, and let $n$ be the denominator of $\trd(\omega\beta)$, it will not take too long before a choice of $\beta$ such that the discriminant of the order $\ZZ\langle \omega,n\beta\rangle$ can be efficiently factored into primes.\footnote{Aside from the fact that the discriminant will be divisible by $p^2$ (since any order is contained in a maximal order), we expect it to behave like a ``random integer'' in some sense, and easily-factorable integers are not too rare in the range of values that appear to arise in practice.}
\end{enumerate}

\begin{lem}\label{OD}
    Take assumptions~\ref{normeq} and~\ref{easilyfactor}. Given any fundamental discriminant $-4M\leq D<0$ with $\left(\frac{D}{p}\right)=-1$, a maximal order of $B$ containing an integral element $\alpha$ with $\nrd(\alpha)\leq M$ and $\QQ(\alpha)\cong\QQ(\sqrt{D})$ can be computed efficiently.
\end{lem}
\begin{proof}
    For $D$ satisfying the above conditions, there is an embedding of $K=\QQ(\sqrt{D})$ into $B$ by~\cite[Proposition 14.6.7]{voight_book}; this implies that $B\otimes_\QQ K$ is split~\cite[Lemma 5.4.7]{voight_book}, which implies by Theorem 5.4.6(vi) that there is a solution $\N_{K(i)/K}(z)=-p$ for some $z\in K[i]^\times$. Using assumption~\ref{normeq}, we can solve for 
    \[z=(x+y\sqrt{D})+i(z+w\sqrt{D}),\qquad x,y,z,w\in\QQ,\]
    in the norm equation, giving 
    \[(x+y\sqrt{D})^2+q(z+w\sqrt{D})^2=-p.\]
    After multiplying through by $pq$ we have
    \[p^2q+(qz+qw\sqrt{D})^2p+(x+y\sqrt{D})^2pq=0.\]
    Setting $\gamma=pi+qzj+xk$ and $\delta=qwj+yk$, we will have $\trd(\gamma\delta^{-1})=0$ and $\nrd(\gamma\delta^{-1})=-D$ by the proof of Lemma 5.4.7 in~\cite{voight_book}, so that $\sqrt{D}\mapsto \gamma\delta^{-1}$ defines an embedding $\QQ(\sqrt{D})\hookrightarrow B$.
    
    Take $\alpha$ to be whichever of $\frac{\sqrt{D}}{2}$ or $\frac{1+\sqrt{D}}{2}$ is integral (depending on whether $D\equiv 0$ or $1\pmod 4$), considered now as an element of $B$. Take a random integral element $\beta\in B$ such that $\{1,\alpha,\beta\}$ is linearly independent. Setting $n$ to be the denominator of $\trd(\alpha\beta)$, $\ZZ\langle \alpha,n\beta\rangle$ will be an order in $B$. If the discriminant of this order can be factored into primes (by assumption~\ref{easilyfactor}, this can be done after relatively few tries for $\beta$), we can efficiently compute a maximal order $\QuatO$ containing this using Proposition 4.3.4 of~\cite{voight_thesis}. Noting that $\alpha$ has norm at most $M$, $\QuatO$ is the desired maximal order.
\end{proof}

\subsection{Algorithms for Computing Isogenies}

In order to compute isogenies, we will need to use modular polynomials.

\begin{defn}
	The \textbf{$n^{\text{th}}$ modular polynomial} $\Phi_n(x,y)\in\ZZ[x]$ is characterized by the following property: $\Phi_n(j_1,j_2)=0$ if and only if there is a degree $n$ cyclic isogeny $E_{j_1}\to E_{j_2}$ (i.e., an isogeny with a cyclic group as its kernel).
\end{defn}

Modular polynomials are symmetric in $x$ and $y$ ($\Phi_n(x,y)=\Phi_n(y,x)$), and if $n$ is prime, then the degree of each variable in $\Phi_n(x,y)$ is $n+1$. The largest coefficient of $\Phi_n(x,y)$ grows faster than $n^{6n}$~\cite{cohen}, which makes even the storage (let alone the computation) of modular polynomials very difficult as $n$ grows large; for instance it takes more than a gigabyte to store the binary representation of $\Phi_{659}$, and $30$ terabytes to store $\Phi_{20011}$~\cite[1201, 1228]{modpolyisogvolc}. However, it is possible to compute $\Phi_n(x,y)\pmod p$ directly, without first computing it with integer coefficients; for instance, an algorithm given by Br\"{o}ker et. al. computes $\Phi_\ell(x,y)\pmod p$ for $\ell$ prime (the only case we will need) in time $O(\ell^{3+\varepsilon})$~\cite[Theorem 1]{modpolyisogvolc}.
    
Say a fundamental discriminant $D$ is \textbf{valid} if $-4M\leq D<0$ and $\left(\frac{D}{p}\right)=-1$. For each valid fundamental discriminant $D$, let $T_D$ be as in Theorem \ref{clusterthm} (defined in Section~\ref{partitionexist}). Let $E_D$ be the set of pairs $(j,j')\in T_D\times T_D$ such that there is an isogeny $j\to j'$ or $j\to j'^p$ of prime degree at most $\frac{4}{\pi}\sqrt{M}$; Theorem \ref{clusterthm} implies that the graph $(T_D,E_D)$ is connected.

Using these definitions, we can apply Algorithm \ref{precomp} to compute the sets $T_D$, the edges $E_D$, and a specified $j_D\in T_D$ with known endomorphism ring $\End(E_{j_D})$. Proposition~\ref{msmallphi} guarantees that the algorithm correctly builds the set $\SpM$ of supersingular $M$-small curves.

\begin{algorithm}
    \caption{Precomputing the $M$-small partition and a selected curve in each subset.}
    \label{precomp}
    \SetAlgoLined
    \DontPrintSemicolon
    \SetKwInOut{Input}{Input}\SetKwInOut{Output}{Output}
    \Input{$p$ and $M$.}
    \Output{For each valid fundamental discriminant $D$, output $T_D$, $E_D$, a specified $j_D\in T_D$, and a maximal order $\QuatO_D\subseteq B$ isomorphic to $\End(E_{j_D})$.}
    
    Compute $\Phi_\ell(x,y)\pmod p$ for all prime $\ell\leq\frac{4}{\pi}\sqrt{M}$~\cite[Theorem 1]{modpolyisogvolc}.\;
    Initialize an empty list $\SpM$.\;
    \For{$-4M\leq d<0$, $d\equiv 0\text{ or }1\pmod 4$, $\left(\frac{d}{p}\right)=-1$} {
        Compute $H_\quadO(x)$, where $\quadO$ is the quadratic order of discriminant $d$~\cite[Theorem 1]{suth_classpoly}.\;
        Append all $j\in\FF_{p^2}$ satisfying $H_\quadO(j)=0$ to $\SpM$.\label{modpolrootfind} 
    }
    \For {valid fundamental discriminants $D$} {
        Compute a maximal order $\QuatO_D$ that has some $\alpha\in\QuatO_D-\ZZ$ with $\nrd(\alpha)\leq M$ and $\QQ(\alpha)\cong\QQ(\sqrt{D})$ (Lemma~\ref{OD}).\label{makemaxorder}\;
        Compute $j_D\in\FF_{p^2}$ such that $\End(E_{j_D})\cong\QuatO_D$~\ref{obs3}\label{stepobs3}.\;
        Initialize queue $Q_D:=(j_D)$ and empty list $E'_D$. Set $j:=j_D$.\;
        \While{$j\in Q_D$\label{startedgefinding}}{
            \For{prime $2\leq \ell\leq\frac{4}{\pi}\sqrt{M}$}{
                \For{$j'\in\SpM$ such that $\Phi_\ell(j',j)=0\pmod p$ or $\Phi_\ell(j',j^p)=0\pmod p$\label{testisogenous}} {
                    Append $j'$ to the end of the queue $Q_D$. \;
                    Append $(j,j')$ to $E'_D$.
                }
            }
            Set $j$ to be the next element of the queue $Q_D$. If no such element exists, break.
        }\label{endedgefinding}
        Set $T_D:=Q_D\cup Q_D^p$.\;
        Set $E_D:=\bigcup_{(j,j')\in E'_D}\{(j,j'),(j,j'^p),(j^p,j'),(j^p,j'^p)\}$.\;
    }
    For each valid fundamental discriminant $D$, return $T_D$, $E_D$, $j_D$, and $\QuatO_D$.
\end{algorithm}

\begin{algorithm}
    \caption{Computing isogenies between $M$-small supersingular curves.}
    \label{findisog}
    \SetAlgoLined
    \DontPrintSemicolon
    \SetKwInOut{Input}{Input}\SetKwInOut{Output}{Output}
    \Input{$j_1,j_2\in\SpM$, $j_0\in\SpM\cap\FF_p$, and the output of Algorithm~\ref{precomp}.}
   
    \Output{An isogeny $E_{j_1}\to E_{j_2}$, given as a sequence of $\ell$-isogenies for primes $\ell$.}
    
    Find $D_0,D_1,D_2$ such that $j_i\in T_{D_i}$ for each $i$.
    
    \For (\tcp*[f]{short paths within T\_D}){$i\in\{0,1,2\}$} {
    
        Find a sequence of edges in $E_{D_i}$ connecting $j_i$ to $j_{D_i}$.\;
        By following these edges, compute an isogeny $\phi_{D_i}:E_{j_i}\to E_{j_{D_i}}$ or $\phi_{D_i}:E_{j_i}\to E_{j_{D_i}}^{(p)}$.\label{smallisogfind}
    }
    \For (\tcp*[f]{long paths between T\_D}){$i\in\{1,2\}$} {
        Using $\QuatO_D$ and $\QuatO_{D_0}$ with Observations~\ref{obs1} and~\ref{obs2}, find an isogeny $\Psi_i:E_{j_{D_0}}\to E_{j_{D_i}}$ or $\Psi_i:E_{j_{D_0}}\to E_{j_{D_i}}^{(p)}$.\label{largeisogfind}
        
        Let $\widehat{\phi_{D_i}}$ denote the dual of $\phi_{D_i}$. Choose $\alpha,\beta\in \{1,p\}$ such that the composition $\Gamma_i:=\widehat{\phi_{D_i}}\circ \Psi_i^\alpha\circ \phi_{D_0}^\beta:E_{j_0}\to E_{j_i}$ is defined. \;
    }
    
    Return $\Gamma_2\circ\widehat{\Gamma_1}:E_{j_1}\to E_{j_2}$.
\end{algorithm}

Note that $H_\quadO(x)$ will have degree $O(M^{1/2+\varepsilon})$ (Proposition~\ref{CMsize}), and the polynomials $\Phi_\ell(x,j)$ will have degree $\ell+1=O(M^{1/2})$. Assuming the conditions under which each appear in the algorithm, these polynomials will split in $\FF_{p^2}$, because their roots will be $j$-invariants of supersingular curves. Thus, assuming an oracle for Assumptions \ref{normeq} and \ref{easilyfactor}, and an oracle that finds all roots of a polynomial of degree $O(M^{1/2+\varepsilon})$ that splits over $\FF_{p^2}$, Algorithm~\ref{precomp} can be shown to run in time polynomial in $M$ and $\log p$.

As noted above, if we want to guarantee existence of a path from any $M$-small supersingular curve to any other one (and not just to one out of a conjugate pair), we will need to be able to route isogenies through an $M$-small supersingular curve defined over $\FF_p$. Such a curve should typically be fairly easy to find; the following lemma gives us a condition on $M$ under which such a curve will be guaranteed to exist.

\begin{lem}
    Let $q=-i^2$. If $M\geq q$, then there exists an $M$-small supersingular $j$-invariant in $\FF_p$.
\end{lem}
\begin{proof}
    There is a maximal order $\QuatO$ containing $\{1,i,j,k\}$, which corresponds by the Deuring correspondence to some supersingular $j$-invariant $j$. Since $i\in\QuatO$ and $\nrd(i)=q\leq M$, $j$ is $M$-small. Since $j\in\QuatO$ and $\ZZ[j]\cong\ZZ[\sqrt{-p}]$, we have $j\in\FF_p$~\cite[Proposition 2.4]{delfs}.
\end{proof}

\noindent
Suppose we have completed Algorithm~\ref{precomp}. If we have some $j_0\in\SpM\cap\FF_p$, then we can apply Algorithm \ref{findisog} to compute an isogeny between any two $M$-small supersingular curves $j_1,j_2\in\SpM$. At each step in the algorithm, the isogenies in question may be recorded as a sequence of $\ell$-isogenies for relatively small primes $\ell$ (in particular, $\ell=O(\sqrt{M})$ in step \ref{smallisogfind} by Theorem \ref{clusterthm}, and $\ell=O(\log p)$ in step \ref{largeisogfind} by Observations \ref{obs1} and \ref{obs2}).

Even if we do not have a $j$-invariant $j_0\in\SpM\cap\FF_p$, a modification of Algorithm \ref{findisog} can still produce isogenies between $M$-small supersingular curves. If we obtain an isogeny $E_{j_1}\to E_{j_2}^{(p)}$, we may simply compose this isogeny with the $p^\text{th}$ power Frobenius $E_{j_2}^{(p)}\to E_{j_2}$. However, the resulting isogeny will be inseparable, and will not be expressible as a composition of $\ell$-isogenies for small primes $\ell$.

\subsection{Example}\label{exmple}

It is worth examining how well Algorithm~\ref{precomp} works in practice; in particular, line \ref{makemaxorder} depends on the unproven assumptions~\ref{normeq} and~\ref{easilyfactor}, so we will focus on the time this step takes.

Let $p=2^{256}+297$; we can take $B$ defined by $i^2=-7$ and $j^2=-p$. Let $M=100$. There are $62$ valid fundamental discriminants $D$:
\[-7,-15,-20,-40,-43,-47,-55,-56,-59,-79,-83,-84,-91,-95,\ldots,-399.\]
For each of these $D$, we computed $\QuatO_D$ as in Algorithm~\ref{precomp}, Line~\ref{makemaxorder}. To do this for all valid $D$ took $60$ seconds on a generic personal laptop. In each case, we were able to take $\beta=i$ or $\beta=j$ in Assumption~\ref{easilyfactor}. 

In practice, it seems as though the real bottleneck of Algorithm~\ref{precomp} is the edge-finding algorithm (lines \ref{startedgefinding}--\ref{endedgefinding}); this took $4105$ seconds on the same laptop.
	
\section{Counting $M$-small Curves}\label{appendixcount}

We will estimate the size of various sets of $M$-small curves, starting small and working up to progressively larger sets.

\begin{prop}\label{CMsize}
    Let $\quadO$ have discriminant $-4M\leq \disc\quadO<0$. Let $C_\quadO$ denote the set of isomorphism classes of elliptic curves $E$ over $\overline{\FF_p}$ such that $\quadO$ optimally embeds in $\End(E)$. Then
    \[|C_\quadO|\leq \deg H_{\quadO}(x)=|\Cl(\quadO)|=O(M^{1/2+\varepsilon}).\]
\end{prop}
\begin{proof}
    The first inequality follows from Proposition~\ref{msmallphi} by counting roots, and we have the middle equality $\deg H_\quadO(x)=|\Cl(\quadO)|$ by~\cite[Proposition 13.2]{cox}. Let $\quadO=\quadO_{K,f}$, let $D$ be the discriminant of $K$, and $h(D)$ the class number of $K$. Then
    \[|\Cl(\quadO)|\leq h(D)f\prod_{\text{prime }\ell\mid f}\left(1-\left(\frac{D}{p}\right)\frac1p\right)\]
    using the formula for the class number of nonmaximal orders~\cite[Theorem 7.24]{cox}. We can bound this above by $h(D)\psi(f)$ using the Dedekind $\psi$ function, defined on positive integers as
    \[\psi(n):=n\prod_{\text{prime }\ell\mid n}\left(1+\frac1\ell\right).\]
    We have $\psi(n)=O(n\log\log n)$~\cite[Corollary 3.2]{extremepsi}, and the classical bound $h(D)=O(|D|^{1/2}\log D)$ (for instance, by Dirichlet's class number formula~\cite[\S 6 (15)]{davenport} and bounds of the form $|L(1,\chi_D)|=O(\log D)$~\cite{sound}). Together these give the bound
    \[|\Cl(\quadO)|=O(f|D|^{1/2}\log D\log\log f)=O(M^{1/2+\varepsilon}),\]
    using $f^2|D|=\disc\quadO\leq 4M$.
\end{proof}

\begin{prop}\label{TDsize}
    Let $D$ be a fundamental discriminant, and
    \[T_D:=\{j\in \Sp:\QQ(\alpha)\cong\QQ(\sqrt{D})\text{ for some }\alpha\in\End(E_j)-\ZZ,\deg\alpha\leq M\}\]
    be the set from Theorem~\ref{clusterthm} (defined in Section~\ref{partitionexist}). Then \[|T_D|=O\left(\frac{M\log |D|}{\sqrt{|D|}}\right).\]
\end{prop}

\noindent
The structure of $T_D$ will depend heavily on its relationship to $M$, as the proof will illustrate. If $D$ is very small, then many different quadratic orders optimally embed into endomorphism rings of curves in $T_D$ ($N$ is large), but each quadratic order embeds in only a couple of these endomorphism rings ($h(D)$ is small). If $D$ is comparable to $M$, then there are very few quadratic orders that optimally embed into endomorphism rings of curves in $T_D$ ($N$ is small), but each quadratic order optimally embeds into many different endomorphism rings ($h(D)$ is large). Intuitively, the ``isogeny valley'' $T_D$ is deep and narrow for small $|D|$, but shallow and wide for large $|D|$.

\begin{proof}
    Let $K$ be a field of discriminant $D$, and let $C_D$ be the set of isomorphism classes of maximal orders $\QuatO\subseteq B$ containing an element $\alpha$ with $\nrd(\alpha)\leq M$ and $\QQ(\alpha)\cong K$. By the Deuring correspondence we have $|T_D|\leq 2|C_D|$, so it suffices to count $C_D$.

    Suppose $\alpha\in\QuatO$ has $\nrd(\alpha)\leq M$ and $\QQ(\alpha)\cong K$. We have $\alpha\in\QuatO\cap \QQ(\alpha)\cong \quadO_{K,f}$ for some conductor $f$. We must have $f^2|D|/4\leq\nrd(\alpha)\leq M$, implying that $f\leq\lfloor \sqrt{4M/|D|}\rfloor=N$. Hence, summing over all possible quadratic orders of $K$ with conductors in this range, we have
    \[|C_D|\leq \sum_{f=1}^N |\Cl(\quadO_{K,f})|\leq h(D)\sum_{f=1}^N \psi(f)\]
    using the proof of Proposition~\ref{CMsize}. This value is
    \[h(D)\left(\frac{30M}{\pi^2|D|}+O(N\log N)\right)\]
    by~\cite[Lemma 2.1]{hurlimann}. Applying $h(D)=O(|D|^{1/2}\log |D|)$, we get the desired bound $|T_D|=O\left(M\log |D|/\sqrt{|D|}\right)$.
\end{proof}

\begin{prop}\label{numMsmall}
    The number of $M$-small curves is $O(M^{3/2})$.
\end{prop}

\begin{proof}
    Given an $M$-small order $\QuatO$, let $\alpha\in\QuatO-\ZZ$ have $\nrd(\alpha)\leq M$. Then $\alpha$ is in some quadratic order $\quadO$, and $|\disc \quadO|/4\leq\nrd(\alpha)$ implies $-4M\leq\disc \quadO<0$. For every possible quadratic order, there are at most $|\Cl(\quadO)|$ isomorphism classes of maximal orders in which $\quadO$ is optimally embedded, meaning that we obtain an upper bound for the number of $M$-small maximal orders by summing $|\Cl(\quadO)|$ over all quadratic orders with $-4M\leq\disc \quadO<0$.
    
    A quadratic order $\quadO$ is uniquely determined by its discriminant, and there is a bijection between $\Cl(\quadO)$ and the set of reduced primitive positive-definite binary quadratic forms of discriminant $\disc \quadO$ (Theorem 7.7(ii) and Theorem 2.8 of \cite{cox}). That is, it suffices to bound the number of triples $(a,b,c)\in\ZZ^3$ with $-a<b\leq a\leq c$ and $b\geq 0$ if $a=c$, $\gcd(a,b,c)=1$, and $-4M\leq b^2-4ac<0$. 
    
    From $|b|\leq a\leq c$, we have $-4M\leq b^2-4ac\leq -3a^2$, so $a\leq \sqrt{4M/3}$. Likewise $-4M\leq b^2-4ac\leq a^2-4ac$ implies $a\leq c\leq \frac{a}{4}+\frac{M}{a}$. Together with $-a<b\leq a$ we conclude that there are at most
    \[\left(\frac{a}{4}+\frac{M}{a}-a+1\right)(2a)\leq 2M+1\]
    valid pairs $(b,c)$ for a given $a$; summing over the $\sqrt{4M/3}$ options for $a$ gives $O(M^{3/2})$ triples.
\end{proof}

\begin{figure}
    \centering
    
    \begin{tikzpicture}
    \begin{axis}[
        ybar,
        axis lines=left,
        ylabel=Number of primes $p$,
	    xlabel=Proportion of $M$-small curves which are supersingular mod $p$,
        ymin=0,
        width=0.8\textwidth,
        height=0.5\textwidth
    ]
    \addplot +[
        hist={
            bins=70,
            data min=0.15,
            data max=0.85
        },
        blue!80!black,
        fill=blue!80!black
    ] table [y index=0] {Fig2data.txt};
    \end{axis}
    \end{tikzpicture}
    
    \caption{A histogram (bins of width $0.01$) of the proportion of $100$-small curves that are supersingular mod $p$, as $p$ varies over $1000$ consecutive primes $2^{40}<p\leq 2^{40}+27201$. Discussed in Remark \ref{halfsupersing}.
    Data computed using Magma~\cite{magma}.}
    \label{fig:percentsupsing}
\end{figure}
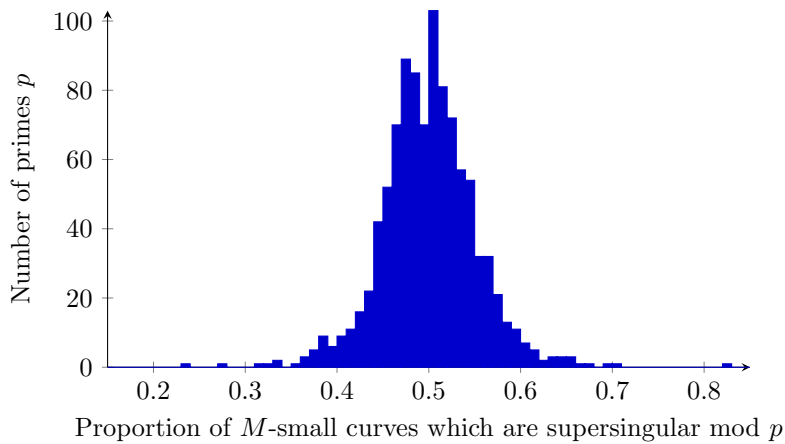

\begin{rmk}\label{halfsupersing}
    When $M\ll p$, we observe that roughly half of all $M$-small curves are supersingular; for instance, with $p=2^{256}+297$ and $M=100$ (the example discussed in Section \ref{exmple}), there are $1108$ $M$-small curves, of which $528$ (about $0.48$ of the total) are supersingular. In Figure~\ref{fig:percentsupsing}, we see that the proportion of $M$-small curves that are supersingular appears to follow a distribution centered at $0.5$; for $94\%$ of the primes $p$ considered, between $0.4$ and $0.6$ of $M$-small curves were supersingular modulo $p$.
    
    Heuristically, this follows from the observation that a root of $H_\quadO(x)\pmod p$ is supersingular if and only if $p$ does not split in the field of fractions of $\quadO$ (Proposition \ref{msmallphi}). For each quadratic order $\quadO$, the set of primes which split in the field of fractions of $\quadO$ have density $\frac12$ (Chebotarev's Density Theorem), so for a set of quadratic orders with discriminants in a given range, we might expect that $p$ will split in the field of fractions of about half of them.
    
    This observation clearly fails for $M$ large enough, because there are only finitely many supersingular curves, but infinitely many ordinary ones. This is because supersingular curves have complex multiplication by infinitely many distinct quadratic orders; that is, even though half (i.e. infinitely many) of the polynomials $H_\quadO(x)\pmod p$ should have supersingular roots, each individual supersingular $j$-invariant will be a root of infinitely many of them. But for small enough values of $M$, at most one of these quadratic orders can have $-4M\leq \disc\quadO<0$, by Theorem \ref{subfieldbound}. So for $M\ll p$, we expect the set roots of $H_\quadO(x)\pmod p$ for $-4M\leq \disc\quadO<0$ to have similar numbers of ordinary and supersingular curves.
\end{rmk}

\begin{prop}\label{allcurvesMsmall}
	All supersingular $j$-invariants are $\left(\frac12 p^{2/3}+\frac14\right)$-small. The exponent is the best possible: if $\theta<\frac23$ then for any constant $C$, there exists a prime $p$ and a supersingular $j$-invariant mod $p$ which is not $\left(Cp^\theta\right)$-small.
\end{prop}

\noindent
The sufficiency of $\frac23$ was noted by Elkies~\cite[Section 4]{elkies}, and Yang showed that no smaller exponent could be taken~\cite[Proposition 1.1]{yang}. The proof given here roughly follows each of their approaches. Notice that Elkies' bound uses the ``large-scale'' structure of maximal orders, namely the geometry of the full $4$-dimensional lattice, while Yang's bound uses the ``small-scale'' structure, counting embedded quadratic orders of small discriminant.
    
\begin{proof}
    We can embed $B$ into $\RR^4$ as follows:
    \[a+bi+cj+dk\mapsto (a, b\sqrt{q}, c\sqrt{p}, d\sqrt{qp}).\]
    This makes the reduced norm $(a+bi+cj+dk)\mapsto a^2+qb^2+pc^2+qpd^2$ agree with the standard Euclidean norm on $\RR^4$. A maximal order $\QuatO\subseteq B$ will be a $4$-dimensional lattice of covolume $\frac{p}{4}$ under this embedding~\cite[(2.2)]{chevgal}. Projecting $\QuatO$ onto the orthogonal complement of $1$ gives a $3$-dimensional lattice of covolume $\frac p4$. By Theorem II.III.A of~\cite{cassels}, any such lattice must have a nonzero element $v$ with length
    \[|v|\leq\left(\frac p4\sqrt{2}\right)^{1/3}=\frac{p^{1/3}}{\sqrt{2}}.\]
    An element of $\QuatO$ that projected onto $v$ must be of the form $\frac k2+v$ for some integer $k$, because the reduced trace of an integral element is an integer. Hence either $v\in\QuatO$ or $\frac 12+v\in\QuatO$, and the reduced norm is either $\frac12p^{2/3}$ or $\frac12 p^{2/3}+\frac14$. This shows $\QuatO$ is $\left(\frac12 p^{2/3}+\frac14\right)$-small.
    
    Conversely, we saw that the number of $M$-small curves is $O(M^{3/2})$, by summing sizes of ideal class groups of embedded quadratic orders (Proposition~\ref{numMsmall}). So if $\theta<\frac{2}{3}$ then the number of $(Cp^\theta)$-small curves will be $O(p^{3\theta/2})$, with $\frac{3\theta}{2}<1$. But the number of supersingular curves is $\frac p{12}+O(1)$~\cite[Theorem V.4.1(c)]{silverman}, which grows faster than the set of $(Cp^\theta)$-small curves.
\end{proof}

\section{Prime-to-$\ell$ isogenies repel length-$\ell$ vertical steps}\label{diffprimeisog}

For this appendix, we assume the setup of Sections~\ref{quatbackground}--\ref{embeddingorders}.

Recall Lemma~\ref{linkdown}, which states that given any maximal order $\QuatO$ with $\ZZ[\ell\beta]$ optimally embedded, there is a maximal order $\QuatO'$ with $\ZZ[\beta]$ optimally embedded which is distance $\ell$ away. If we replace $\QuatO'$ with an isomorphic order $\QuatO''$, the following proposition and corollary show that $d(\QuatO,\QuatO'')$ must be either a multiple of $\ell$ or extremely large. This indicates why we must consider all primes in order to find a short path in Proposition~\ref{opath} and Theorem~\ref{clusterthm}(b).

\begin{prop}\label{longifexcludel}
    Let $\ell$ be a prime, and $\beta\in\quadO_K$. Suppose maximal orders $\QuatO$ and $\QuatO'$ have $\ZZ[\ell\beta]$ and $\ZZ[\beta]$ optimally embedded, respectively. If $d(\QuatO,\QuatO')$ is not divisible by $\ell$, then $d(\QuatO,\QuatO')\geq \frac{p}{4\ell\nrd(\beta)}$.
\end{prop}

\begin{proof}
    If the optimal embeddings of $\ZZ[\ell\beta]$ and $\ZZ[\beta]$ were to land in the same subfield $K\subseteq B$, then $|\QuatO':\QuatO\cap \QuatO'|$ would be divisible by $\ell$, a contradiction. Hence we must have $\QuatO\cap K\cong \ZZ[\ell\beta]$ and $\QuatO'\cap K'\cong \ZZ[\beta]$ for distinct (but isomorphic) fields $K$. let $\quadO:=\QuatO\cap K$ and $\quadO':=\QuatO\cap K'$ both be optimally embedded in $\QuatO$. Since $K$ and $K'$ are isomorphic but distinct, Theorem~\ref{subfieldbound} tells us that $\disc\quadO\disc\quadO'\geq p^2$. 
    
    Now $\ell\beta\in \quadO$ and $d\beta\in\quadO'$, so as in the proof of Proposition~\ref{largeinters}, we can conclude that
    \[d^2\geq\frac{\disc\quadO}{4\ell^2\nrd(\beta)}\frac{\disc\quadO'}{4\nrd(\beta)}\geq \frac{p^2}{16\ell^2\nrd(\beta)^2}.\qedhere\]
\end{proof}

\begin{cor}\label{isogwithoutl}
    Let $\ell$ be a prime, $M\in\ZZ$, and $E$ an $(M/\ell^2)$-small supersingular curve over $\FF_{p^2}$. Then there exists an $M$-small supersingular curve $E'$ over $\FF_{p^2}$ connected to $E$ by an $\ell$-isogeny, such that if $\phi:E\to E'$ is any isogeny with degree relatively prime to $\ell$, then 
    \[\deg\phi\geq \frac{p\ell}{4M}.\]
\end{cor}

\begin{proof}
    For some imaginary quadratic field $K$ and some $\beta\in K$ with norm at most $M/\ell^2$, the quadratic order $\ZZ[\beta]$ is optimally embedded in $\End(E)$. Modifying the proof of Lemma~\ref{linkdown}, we can find a maximal order $\QuatO'$ with $\ZZ[\ell\beta]$ optimally embedded, and such that $d(\End(E),\QuatO')=\ell$. By the Deuring correspondence, we obtain an $M$-small curve $E'$ connected to $E$ by an $\ell$-isogeny. 
    
    Now any isogeny $\phi:E\to E'$ corresponds to an ideal $I$ linking $\End(E)$ to some maximal order $\QuatO''\cong\End(E')$. In particular, $I$ must be contained in $\QuatO\cap\QuatO''$, so if $\deg\phi=\nrd(I)$ is not divisible by $\ell$, then neither is $d(\End(E),\QuatO'')$. Hence, by Proposition~\ref{longifexcludel}, 
    \[\deg\phi=\nrd(I)\geq d(\End(E),\QuatO'')\geq \frac{p}{4\ell \nrd(\beta)}\geq\frac{p\ell}{4M}.\qedhere\]
\end{proof}

\end{document}